\documentclass[10pt]{article}
\usepackage{amssymb}
\usepackage{graphicx}
\usepackage{xcolor} 
\usepackage{tensor}
\usepackage{fullpage} 
\usepackage{amsmath}
\usepackage{amsthm}
\usepackage{verbatim}
\usepackage{hyperref}
\usepackage{enumitem}
\setlist[enumerate]{leftmargin=1.5em}
\setlist[itemize]{leftmargin=1.5em}

\usepackage{seqsplit}

\definecolor{green}{rgb}{0,0.8,0} 

\newtheorem{theorem}{Theorem}[section]

\newtheorem{lemma}[theorem]{Lemma}

\theoremstyle{definition}

\theoremstyle{remark}
\newtheorem{remark}[theorem]{Remark}

\numberwithin{equation}{section}
\newcommand{\nrm}[1]{\Vert#1\Vert}

\newcommand{\brk}[1]{\langle#1\rangle}

\newcommand{\nnrm}[1]{{\vert\kern-0.25ex\vert\kern-0.25ex\vert #1 
    \vert\kern-0.25ex\vert\kern-0.25ex\vert}}

\renewcommand{\Re}{\mathrm{Re}}

\newcommand{\lap}{\Delta}

\newcommand{\rd}{\partial}
\newcommand{\nb}{\nabla}

\newcommand{\alp}{\alpha}

\newcommand{\gmm}{\gamma}

\newcommand{\eps}{\epsilon}

\newcommand{\Lmb}{\Lambda}

\newcommand{\Omg}{\Omega}

\newcommand{\bbR}{\mathbb R}

\newcommand{\bbT}{\mathbb T}

\newcommand{\calL}{\mathcal L}

\newcommand{\calS}{\mathcal S}
\newcommand{\calT}{\mathcal T}

\newcommand{\whr}{\widehat{\rho}}			 
\newcommand{\owhr}{\overline{\widehat{\rho}}}			 

\newcommand{\tu}{\tilde{u}}					

\newcommand{\trho}{\tilde{\rho}}

\newcommand{\R}{\mathbb R}

\newcommand{\e}{\varepsilon}

\newcommand{\bq}{\begin{equation}}
\newcommand{\eq}{\end{equation}}

\begin{document}

\title{Classical solutions for fractional porous medium flow}
\author{Young-Pil Choi\thanks{Department of Mathematics, Yonsei University, Seoul 03722, Republic of Korea. E-mail: ypchoi@yonsei.ac.kr} \and In-Jee Jeong\thanks{Department of Mathematics, Korea Institute for Advanced Study, Seoul 02455, Republic of Korea. E-mail: ijeong@kias.re.kr}}
\date{\today}



\maketitle


\begin{abstract}
	We consider the fractional porous medium flow introduced by Caffarelli and Vazquez (\cite{CaVa}) and obtain local in time existence, uniqueness, and blow-up criterion for smooth solutions.  The proof is based on establishing a commutator estimate involving fractional Laplacian operators.  
\end{abstract} 


\section{Introduction}

\subsection{Main result}

In this paper, we consider the Cauchy problem for the fractional porous medium flow, possibly with dissipation:
\begin{equation}  \label{eq:Euler-first}
\begin{aligned}
&\rd_t \rho + \nb\cdot (\rho u) = \nu \lap\rho, \quad  u = c_K\Lmb^{\alp-d} \nb\rho 
\end{aligned} 
\end{equation} for $\rho(t,\cdot):\Omg \rightarrow \bbR_+$ where $\Omg = \bbR^d$ or $\bbT^d$. Here, $\nu \ge 0$, $c_K\in\bbR$, and $-2 \le \alp-d \le 0$ are parameters, and $\Lmb^s$ is the $s-$fractional power of $\Lmb:=(-\lap)^{\frac{1}{2}}$, to be defined precisely below. Our main result is local well-posedness of classical solutions for \eqref{eq:Euler-first}, under appropriate assumptions on the parameters.
\begin{theorem}\label{thm:lwp}
	The fractional porous medium equation with dissipation is locally well-posed for $-2\le \alp-d \le 0$ in each of the following cases:
	\begin{enumerate}
		\item (Repulsive and inviscid case) For $c_K < 0$ and $\nu=0$, \eqref{eq:Euler-first} is locally well-posed in $L^1\cap H^s(\Omg)$ with $s> \frac{d}{2}+3$. More precisely, given any non-negative initial data $\rho_0 \in L^1\cap H^s(\Omg)$, there exists $T=T(\rho_0)>0$ and a unique non-negative solution $\rho \in C([0,T);L^1\cap H^s(\Omg))$ of \eqref{eq:Euler-first} with $\rho(t=0)=\rho_0$. 
		\item (Viscous case) For $c_K \in \bbR$ and $\nu>0$, \eqref{eq:Euler-first} is locally well-posed in $L^1\cap H^s(\Omg)$ with $s> \frac{d}{2}+2$; for any $\rho_0 \in L^1\cap H^s(\Omg)$, there exists $T=T(\rho_0)>0$ and a unique solution $\rho \in C([0,T);L^1\cap H^s(\Omg))$. In the case $\alp-d=0$, $c_K>0$, we need an additional smallness condition on the initial data; $\nrm{\rho_0}_{L^\infty} < c\nu/c_K$, where $c>0$ is an absolute constant.
	\end{enumerate}
 We have the following blow-up criteria. The unique local solution in $L^\infty([0,T);H^s(\Omg))$ can be continued past $T>0$ if and only if \begin{equation*}
\begin{split}
\limsup_{t<T}\int_0^t \nrm{ |\xi|^2 ( 1 + |\xi|)\whr(\tau,\xi) }_{L^1_\xi} \, d\tau < \infty
\end{split}
\end{equation*} and
\begin{equation*}
\begin{split}
\limsup_{t<T}\int_0^t \nrm{ |\xi|( 1 + |\xi|)\whr(\tau,\xi) }^2_{L^1_\xi} \, d\tau < \infty. 
\end{split}
\end{equation*} in Cases 1 and 2, respectively. Here $\whr(\tau,\cdot)$ is the Fourier transform of $\rho(\tau,\cdot)$. In particular, in both cases, if $\rho_0\ge0$ and $\rho_0 \in L^1 \cap (\cap_{s\ge 0} H^s)(\Omg)$, then the unique solution belongs to $\rho \in C^\infty([0,T)\times\Omg)$ for some $T>0$. 
\end{theorem}

\subsection{Previous works}

The \textit{fractional porous medium flow} \begin{equation}\label{eq:FPME}
\begin{split}
\rd_t\rho + \nb\cdot(\rho u) = 0,\quad u = -\Lmb^{ \alp-d}\nb\rho 
\end{split}
\end{equation} was introduced in \cite{CaVa} and further studied in \cite{CSV}, exactly in the range $-2<\alp-d<0$. As we shall see later (\ref{subsec:remarks}), both endpoints $\alp-d = -2$ and $\alp-d = 0$ are \textit{critical}, for different reasons. The parameter $\alp$ dictates the decay rate of the interaction kernel $\Lmb^{ \alp-d}$, which is simply a constant multiple of $|x|^{-\alp}$. This allows one to consider long-range interactions. The authors in \cite{CaVa,CSV} proved the existence of mass preserving weak solutions, finite propagation property, boundedness and $C^\alpha$ regularity for non-negative and integrable initial data $\rho_0$. There was no mention on the existence of strong (or classical) solutions in \cite{CaVa,CSV}.  {
More recently, the existence of weak solutions and other properties such as entropy, regularizing effect, and decay estimates were established in \cite{LMS18} by employing a gradient flow approach.

Time-asymptotic behaviors in terms of its self-similar solutions are investigated in \cite{CV11, CHSV15}. The self-similar profile, so called {\it fractional Barenblatt profile}, was characterized in \cite{CV11} by solving an elliptic obstacle problem. In \cite{CHSV15}, the one--dimensional case was considered, and  exponential convergence of solutions in self-similar variables to the unique stationary state (explicitly constructed in \cite{BIK15}) is obtained. 
}

Rather recently, there has been a growing interest in the general systems of the form 
{
\begin{equation}\label{eq:general}
\begin{split}
\rd_t\rho + \nb\cdot(\sigma(\rho) u) = 0,\quad u = - \nb \calL[\rho]
\end{split}
\end{equation} where $\sigma$ and $\calL$ are given by possibly non-linear and/or non-local operators. The equation \eqref{eq:general} describes the propagation of the particle density $\rho(t,\cdot)$ in a medium, with interactions governed by $\calL$ and mobility $\sigma$ of the reference system. Motivations for studying \eqref{eq:general}  for various $\calL$, which include phase separation of lattice active matter, granular flow, biological swarming, pattern formation,  can be found in the works \cite{CCCSS16,CMV03, GL1,GL2,GL3,GL4, LT04, ME99,STV19}. We finally refer to \cite{CCY19,Va17} and the references therein for a  general discussion on nonlinear drift-diffusion models.
}

While previous mathematical works were mainly focused on the regularity of suitably defined weak solutions, it is natural to ask whether higher spatial regularity propagate in time, with smooth initial data. Moreover, the existence of smooth solutions seems to be a necessary condition for a rigorous derivation of \eqref{eq:FPME} from more primitive systems. In our companion work \cite{CJ-limit}, the equation \eqref{eq:Euler-first} was derived rigorously from the damped Euler--Riesz system in the \textit{large friction} limit:
\begin{align}\label{eq:ER}
\left\{
\begin{aligned}
&\rd_t \rho^{(\e)} + \nb \cdot (\rho^{(\e)} u^{(\e)}) = 0, \\
&\rd_t (\rho^{(\e)} u^{(\e)}) + \nabla \cdot ( \rho^{(\e)} u^{(\e)}\otimes u^{(\e)}) +  \frac1\e c_p \nb p(\rho^{(\e)}) =  - \frac1\e \rho^{(\e)} u^{(\e)} + \frac1\e c_K \rho^{(\e)} \nb \Lmb^{\alp-d}\rho^{(\e)},
\end{aligned}
\right.
\end{align} where $p(\rho)=\rho^\gmm$. That is, the sequence of solutions $(\rho^{(\e)},u^{(\e)})$ for \eqref{eq:ER} converges in a sense to \eqref{eq:Euler-first} in the limit $\e\rightarrow0$. Indeed, as $\e\rightarrow 0$, one can formally solve for $u^{(\e)}$ in terms of $\rho^{(\e)}$ from the second equation in \eqref{eq:ER} by keeping only the terms with coefficient $1/\e$, which gives \eqref{eq:Euler-first} without dissipation when $c_p=0$ (pressureless case) or with dissipation of the form $\lap(\rho^\gmm)$ when $c_p>0$. In the proof of convergence, it was essential that we have smooth (at least $u(t,\cdot)\in C^1$) solutions to \eqref{eq:Euler-first}. We refer the interested readers to \cite{CJ-limit} for the precise statement.

\subsection{Remarks}\label{subsec:remarks}

Let us give a few remarks related to the statement of Theorem \ref{thm:lwp}.

\medskip

\noindent \textit{The Riesz operator $\Lmb^s$.}  The operator $\Lmb^s$ for $s\in\bbR$ is defined by the Fourier multiplier with symbol $|\xi|^s$; that is, \begin{equation*}
\begin{split}
\widehat{\Lmb^s f}(\xi) = |\xi|^s \widehat{{f}}(\xi),
\end{split}
\end{equation*} where the Fourier transform is defined by $\widehat{{f}}(\xi) = \int_{\Omg} f(x) e^{-ix\cdot\xi} dx$. For $\Lmb^s f$ to be well-defined as a bounded function for $s<0$, we need $f$ to have zero mean in the case of $\bbT^d$ and have some decay at infinity in the case of $\bbR^d$, respectively. This is the reason for assuming $L^1$ of $\rho$ in Theorem \ref{thm:lwp}: for $d\ge 1$ and $-2\le \alp-d\le 0$, we have $\nrm{ \Lmb^{\alp-d} \nb\rho}_{L^\infty} \le C_{\alp,d}\nrm{ \rho}_{L^1\cap H^s}$ for $ s> \frac{d}{2} + 1$. However, $L^1(\Omg)$ norm of $\rho$ does not enter the a priori estimate in $H^s(\Omg)$, as it only sees $\nabla u$ and its derivatives. 

Note that when $\alp$ becomes small so that $\alp-d<-2$, the kernel decays so slowly that in general we do not have $\Lmb^{\alp-d} \nb\rho\in L^\infty$. The endpoint $\alp-d=-2$ is critical in this regard. 

\medskip

\noindent \textit{Even more singular kernels.} Given local well-posedness in the range $\alp-d\le 0$, one may attempt to extend the same result to the case of more singular kernels, namely when $\alp-d>0$. However, it seems that in the inviscid and repulsive case, the equation is \textit{ill-posed} in Sobolev spaces\footnote{It means that there exist initial data in $H^s$ for large $s$ such that there is no solution in $L^\infty([0,T];H^s)$ for any $T>0$.}. While rigorously establishing ill-posedness could be a challenge, heuristically one can easily see it from the $H^s$-energy estimate (taking $d=1$ for simplicity): \begin{equation*}
\begin{split}
\frac{1}{2}\frac{d}{dt} \int_{\Omg} |\rd^s\rho|^2 \, dx - c_K \int_{\Omg}  \rho |\Lmb^b \nb(\rd^s\rho)|^2\,dx = C_b\int_{\Omg} \lap\rho |\Lmb^b(\rd^s\rho)|^2\, dx + \mbox{lower order terms}.
\end{split}
\end{equation*} More details of this computation can be found at Subsection \ref{subsec:formal} below. Then, one can arrange initial data $\rho_0$ so that the first term on the right hand side dominates the second term on the left hand side, at least at the initial time. 

\medskip

\noindent \textit{The case of degenerate/singular diffusion.} One may consider the following variant of \eqref{eq:Euler-first} with some $\gmm>0$: \begin{equation}  \label{eq:Euler-first-gmm}
			\begin{aligned}
				&\rd_t \rho + \nb\cdot (\rho u) = \nu \lap(\rho^\gmm), \quad  u = c_K\Lmb^{\alp-d} \nb\rho.
			\end{aligned} 
		\end{equation} The convergence theorem in the companion work \cite{CJ-limit} covers \eqref{eq:Euler-first-gmm} as the limiting system, assuming that smooth solutions to \eqref{eq:Euler-first-gmm} exist. Unfortunately, there is a serious difficulty in obtaining $H^s$ estimates in the case $\gmm \ne1$, in terms of the variable $\rho$. This issue disappears when $\rho$ attains a positive lower bound: namely, when there exists a constant $c_0>0$ such that $\nrm{\rho - c_0}_{L^\infty} \le \frac{c_0}{2}$ and $\rho - c_0 \in H^s(\Omg)$. Then, local well-posedness in the viscous case can be extended to $c_K\in \bbR$ and $\gmm\ge 1$ without much difficulty.  

\section{Ideas and Proofs}

Let us emphasize that local well-posedness for \eqref{eq:Euler-first} does not follow directly from simple energy estimates, see Remark \ref{rmk:comm} for more details. In Subsection \ref{subsec:formal}, we provide a heuristic discussion for the proof of local well-posedness with a little bit of pseudo-differential calculus. The actual proof, given in Subsection \ref{subsec:proof}, is completely self-contained and does not require any results from the theory of pseudo-differential calculus.

\subsection{Formal discussion}\label{subsec:formal}

We consider the $H^m$-estimate for a solution of \eqref{eq:Euler-first}:
\begin{equation*}
\begin{split}
\frac{1}{2} \frac{d}{dt} \int_{\Omg} |\rd^m\rho|^2\,dx + \nu\int_{\Omg} |\rd^m\nb\rho|^2 \,dx =  \int_{\Omg} \nb ( \rd^m\rho  )\cdot \rd^m (\rho u) \, dx . 
\end{split}
\end{equation*} We need estimate the right hand side in terms of Sobolev norms for $\rho$. We inspect the terms which are potentially problematic; those arise when either all or almost all derivatives hit $u$ in the expression $\rd^m(\rho u)$. 
\begin{itemize}
	\item Principal term: with $b = (d-\alp)/2 \ge 0$, \begin{equation*}
	\begin{split}
	c_K\int_{\Omg} \rho\nb(\rd^m\rho) \cdot \Lmb^{\alp-d} \nb(\rd^m\rho)  \, dx = c_K \int_{\Omg} \rho |\Lmb^{-b}\nb (\rd^m\rho)|^2 \, dx + c_K \int_{\Omg} [ \Lmb^{-b}, \rho ](\nb\rd^m\rho) \cdot \Lmb^{-b} \nb(\rd^m\rho)\, dx .
	\end{split}
	\end{equation*}
	When $c_K >0$ and $\nu = 0$, we expect the system to be ill-posed. On the other hand, if $\nu>0$, this principal term could be controlled by $\int_{\Omg} |\rd^m\nb\rho|^2\,dx$ when $b>0$. However, for $b = 0$, we need smallness of $\rho$. We now move on to the inviscid case, assuming $c_K<0$. Then, the commutator term can be bounded in terms of $\rho$ in $H^m$ when $b\ge 1$, and simply vanishes for $b = 0$. When $0<b<1$, we recall from pseudo-differential calculus that (assuming that $\rho \in C^\infty$ and decays fast at infinity) \begin{equation*}
	\begin{split}
	\sigma([\Lmb^{-b},\rho]) = \frac{1}{i} \nb_x\rho \cdot \nb_\xi |\xi|^{-b} + q(x,\xi),\quad q \in \calS^{-b-2} 
	\end{split}
	\end{equation*} and hence the main term in the commutator $[\Lmb^{-b},\rho]$ is given by $b \nb\rho \cdot \Lmb^{-b-2}\nb$. (Here, $\calS^{a}$ denotes the class of symbols of order $a$; we say $p\in \calS^a$ if $|\rd_x^n\rd_\xi^m p|(x,\xi) \lesssim_{n,m} \brk{\xi}^{a-m}$ with $\brk{\xi}=\sqrt{1+|\xi|^2}$ for all $n,m\ge0$.)  Therefore, \begin{equation*}
	\begin{split}
	&\int_{\Omg} [ \Lmb^{-b}, \rho ](\nb\rd^m\rho) \cdot \Lmb^{-b} \nb(\rd^m\rho)\, dx \\
	&\quad = b\int_{\Omg} \nb\rho\cdot\Lmb^{-b-2}\nb (\nb\rd^m\rho) \cdot \Lmb^{-b} \nb(\rd^m\rho)\, dx + \int_{\Omg} q(X,D)(\nb\rd^m\rho) \cdot \Lmb^{-b} \nb(\rd^m\rho)\, dx,
	\end{split}
	\end{equation*} and the latter term can be bounded by \begin{equation*}
	\begin{split}
	\left|  \int_{\Omg} q(X,D)(\nb\rd^m\rho) \cdot \Lmb^{-b} \nb(\rd^m\rho)\, dx \right|\le C\nrm{\rho}_{H^{m-b}}^2
	\end{split}
	\end{equation*} since $q \in \calS^{-b-2} $. After an integration by parts, the first term is given by \begin{equation*}
	\begin{split}
	-b\int_{\Omg}  \nb \cdot (\nb\rho \cdot \Lmb^{-b-2}\nb (\nb\rd^m\rho) ) \Lmb^{-b}  \rd^m\rho \, dx & = b\int_{\Omg} \nb\rho \cdot \Lmb^{-b}(\nb \rd^m\rho) \, \Lmb^{-b}(\rd^m\rho) \, dx + o.k. \\
	& = -\frac{b}{2} \int_{\Omg} \lap\rho |\Lmb^{-b}(\rd^m\rho)|^2 \, dx + o.k.
	\end{split}
	\end{equation*}
	
	\item Sub-principal term: \begin{equation*}
	\begin{split}
	c_K\int_{\Omg} \rd\rho  \nb\rd^m\rho \cdot \nb\Lmb^{-2b} \rd^{m-1}\rho \, dx & = -\frac{c_K}{2} \int_{\Omg} \rd^2\rho |\Lmb^{-b}\nb\rd^{m-1}\rho|^2 \, dx \\
	&\quad + c_K \int_{\Omg} [\Lmb^{-b},\rd\rho] (\rd\nb\rd^{m-1}\rho) \Lmb^{-b} (\nb\rd^{m-1}\rho) .
	\end{split}
	\end{equation*} We see that (formally) both terms can be bounded in terms of $\rho$ in $H^m$. There is another sub-principal term, which is \begin{equation*}
	\begin{split}
	\int_{\Omg} \nb(\rd^m\rho) \cdot u \rd^m\rho\,dx. 
	\end{split}
	\end{equation*} However, this is easily estimated after an integration by parts. 
\end{itemize}

\begin{remark}\label{rmk:comm}
	The above discussion is made precise in the proof below. The essence of the argument can be summarized by the following estimate: \begin{equation}\label{eq:comm}
\begin{split}
\nrm{ ([\Lmb^{-b},f \nb] - b(\nb f \cdot\nb )\Lmb^{-b-2}\nb)g }_{L^2} \le C_{d,b,\eps} \nrm{f}_{H^{\frac{d}{2}+3+\eps}} \nrm{g}_{H^{-b-1}}
\end{split}
\end{equation} for any $\eps>0$. This can be viewed as an extension of the following commutator estimate proved in \cite[Proposition 2.1]{CCCGW}: \begin{equation*}
\begin{split}
\nrm{ [\Lmb^{-b},f \nb]  g }_{L^2} \le C_{d,b,\eps} \nrm{f}_{H^{\frac{d}{2}+1-b+\eps}} \nrm{g}_{H^{-b}}.
\end{split}
\end{equation*} 
The authors needed this estimate in order to prove local well-posedness of an active scalar equation, in which the velocity is more singular than the scalar. Indeed, in their model the relation between $u$ and $\rho$ is given by $u = \nb^\perp\Lmb^{\alp-d}\rho$ in 2D with $\nb^\perp = (-\rd_{x_2},\rd_{x_1})$. However, since this velocity is divergence-free, one obtains a cancellation in the expression $\nb\cdot(\rho u)$ for the potentially most delicate term. In this sense, \eqref{eq:Euler-first} is more singular by order 1, which is basically the reason we need to extract the next-order term from the commutator $[\Lmb^{-b},f\nb]$. 
\end{remark}

\subsection{Proof}\label{subsec:proof}

\begin{proof}[Proof of Theorem \ref{thm:lwp}]
	
	We begin with obtaining a priori estimates for the solution of \eqref{eq:Euler-first}, in each of the cases 1 and 2. In the proof, we set $0 \le  b = (d-\alp)/2 \le 1$ and focus on the case $0<b<1$. The endpoint cases $b = 0$ and $b = 1$ can be treated separately without any additional difficulties.
	
	\medskip
	
	\noindent\textit{Case 1. $c_K<0$ and $\nu=0$.}
	
	\medskip 
	
	\noindent Let us consider the case $\Omg = \bbR^d$. The proof readily extends to the $\bbT^d$ case. We fix some $s > \frac{d}{2} + 3$ and compute \begin{equation*}
	\begin{split}
	\frac{d}{dt} \frac{c_d}{2} \nrm{\rho}_{\dot{H}^s}^2 & = - \Re \int_{\bbR^d} |\xi|^s\owhr (\xi) \, |\xi|^s \int_{\bbR^d} i\xi\cdot \widehat{u}( \eta) \whr(\xi-\eta) \, d\eta \, d\xi \\
	& =  c_K \Re \iint_{\bbR^d \times \R^d}   |\xi|^s\owhr (\xi) \, |\xi|^s \xi\cdot\eta |\eta|^{-2b} \whr (\eta)  \whr(\xi-\eta) \, d\eta \, d\xi,
	\end{split}
	\end{equation*} where we have simply used the definition of $u$ and $b$. (Here, $c_d>0$ is an absolute constant coming from the definition of the Fourier transform and $\dot{H}^s$-norm.) For convenience, we shall define the operator $\calT$ as follows: \begin{equation*}
	\begin{split}
	\calT[ G(\xi,\eta) ] = \Re \iint_{\R^d \times \R^d}  G(\xi,\eta) \owhr (\xi)  \whr (\eta)  \whr(\xi-\eta) \, d\eta \, d\xi.
	\end{split}
	\end{equation*} An important observation is that if $G$ is real and anti-symmetric, that is, $G(\eta,\xi)=-G(\xi,\eta)$ for all $\xi,\eta \in \bbR^d$, we have that $\calT[G] = 0$. To see this, we simply make a change of variables $(\eta,\xi)\mapsto (\xi,\eta)$: \begin{equation*}
	\begin{split}
	\calT[G] & = \Re\iint_{\R^d \times \R^d}   G(\eta,\xi) \owhr (\eta)  \whr (\xi)  \whr(\eta-\xi) \, d\eta \, d\xi \\
	& = \Re \iint_{\R^d \times \R^d}  -G(\xi,\eta) \owhr (\eta)  \whr (\xi)  \whr(\eta-\xi) \, d\eta \, d\xi  \\
	& = \Re \iint_{\R^d \times \R^d}  - \overline{ G(\xi,\eta) \owhr (\eta)  \whr (\xi)  \whr(\eta-\xi) } \, d\eta \, d\xi \\
	& = \Re \iint_{\R^d \times \R^d} -G(\xi,\eta) \owhr (\xi)  \whr (\eta)  \whr(\xi-\eta) \, d\eta \, d\xi = - \calT[G].
	\end{split}
	\end{equation*} In the above, we have used anti-symmetry of $G$ and $\overline{\whr(\eta-\xi)} = \whr(\xi-\eta)$. 
	
	\medskip

	\noindent From now on, we shall consider several cases of $G$. 
	\begin{itemize}
		\item $G_s = |\xi|^s |\xi-\eta|^s \xi\cdot\eta |\eta|^{-2b}$. In this case, it is convenient to make a change of variables $\eta = \xi-\mu$ (for fixed $\xi$): then, \begin{equation*}
		\begin{split}
		\calT[G_s] & = \Re \iint_{\R^d \times \R^d} |\xi|^s \overline{\whr(\xi)} |\mu|^s \whr(\mu) \xi\cdot(\xi-\mu)|\xi-\mu|^{-2b} \whr(\xi-\mu) \, d\mu\, d \xi \\
		& = \frac{1}{2}\Re\iint_{\R^d \times \R^d} |\xi|^s \overline{\whr(\xi)} |\mu|^s \whr(\mu) |\xi-\mu|^{2-2b} \whr(\xi-\mu) \, d\mu\, d \xi , 
		\end{split}
		\end{equation*} since \begin{equation*}
		\begin{split}
		\Re \iint_{\R^d \times \R^d} |\xi|^s \overline{\whr(\xi)} |\mu|^s \whr(\mu) (\xi + \mu)\cdot(\xi-\mu)|\xi-\mu|^{-2b} \whr(\xi-\mu) \, d\mu\, d \xi = 0
		\end{split}
		\end{equation*} by the anti-symmetry of the kernel. Then, we can write \begin{equation*}
		\begin{split}
		\xi\cdot (\xi-\mu) = \frac{1}{2}|\xi-\mu|^2 + \frac{1}{2}(\xi+\mu)\cdot(\xi-\mu). 
		\end{split}
		\end{equation*} Therefore, we conclude the bound \begin{equation}\label{eq:TG1}
		\begin{split}
		\left| \calT[G_s] \right| \le C \nrm{ |\xi|^{2-2b}\whr }_{L^1_\xi}\nrm{|\xi|^s\whr}_{L^2}^2.
		\end{split}
		\end{equation} 
		\item $G_0 = |\xi|^s |\eta|^s \xi\cdot\eta |\eta|^{-2b}$. In this case, we rewrite \begin{equation*}
		\begin{split}
		G_0 & = |\xi|^{s-b}|\eta|^{s-2b} ( |\xi|^b - |\eta|^b ) \xi\cdot\eta + (|\xi||\eta|)^{s-b}\xi\cdot\eta \\
		& = |\xi|^{s-\frac{3}{2}b}|\eta|^{s-\frac{3}{2}b} ( |\xi|^b - |\eta|^b ) \xi\cdot\eta + |\xi|^{s-\frac{3}{2}b}|\eta|^{s-2b} ( |\xi|^{\frac{b}{2}} - |\eta|^{\frac{b}{2}} ) ( |\xi|^b - |\eta|^b ) \xi\cdot\eta + (|\xi||\eta|)^{s-b}\xi\cdot\eta.
		\end{split}
		\end{equation*} This gives \begin{equation*}
		\begin{split}
		\calT[G_0] = \calT[G'_0] + c_K\int_{\R^d} \rho | \Lmb^{s-b} \nb\rho|^2\, dx 
		\end{split}
		\end{equation*} with \begin{equation*}
		\begin{split}
		G'_0(\xi,\eta):= |\xi|^{s-\frac{3}{2}b}|\eta|^{s-2b} ( |\xi|^{\frac{b}{2}} - |\eta|^{\frac{b}{2}} ) ( |\xi|^b - |\eta|^b ) \xi\cdot\eta ,
		\end{split}
		\end{equation*} since $|\xi|^{s-\frac{3}{2}b}|\eta|^{s-\frac{3}{2}b} ( |\xi|^b - |\eta|^b ) \xi\cdot\eta $ is anti-symmetric in $(\xi,\eta)$. We write
		\begin{equation*}
		\begin{split}
		|\xi|^b - |\eta|^b = b(\xi-\eta) \cdot \int_0^1 (a\xi + (1-a)\eta) |a\xi + (1-a)\eta|^{b-2} \, da 
		\end{split}
		\end{equation*} so that \begin{equation}\label{eq:b-b}
		\begin{split}
		\left| |\xi|^b - |\eta|^b \right| \le C_b|\xi-\eta|  \max\{ |\xi|^{b-1}, |\eta|^{b-1} \}.
		\end{split}
		\end{equation} Repeating a similar argument for $|\xi|^{\frac{b}{2}} - |\eta|^{\frac{b}{2}}$, we obtain that \begin{equation*}
		\begin{split}
		\left| ( |\xi|^{\frac{b}{2}} - |\eta|^{\frac{b}{2}} ) ( |\xi|^b - |\eta|^b ) \right| \le C_b|\xi-\eta|^2  \max\{ |\xi|^{\frac{3}{2}b-2}, |\eta|^{\frac{3}{2}b-2} \}.
		\end{split}
		\end{equation*} We need to consider separately the cases $|\xi|\le|\eta|$ and $|\xi|>|\eta|$: \begin{equation*}
		\begin{split}
		G_0' = G_0' \mathbf{1}_{ |\xi|\le|\eta |} +  G_0' \mathbf{1}_{ |\xi|>|\eta |}.
		\end{split}
		\end{equation*}  {First, in the case $|\xi|\le|\eta |$, we have 
		\[
		\max\{ |\xi|^{\frac{3}{2}b-2}, |\eta|^{\frac{3}{2}b-2} \} = |\xi|^{\frac{3}{2}b-2},
		\] 
		which gives \begin{equation*}
		\begin{split}
		\left| G_0' \mathbf{1}_{ |\xi|\le|\eta |} \right| (\xi,\eta) \le C_b |\xi|^{s-1}|\eta|^{s+1-2b} |\xi-\eta|^2,
		\end{split}
		\end{equation*} and when $b < \frac{1}{2}$, we can further estimate \begin{equation*}
		\begin{split}
		\left| G_0' \mathbf{1}_{ |\xi|\le|\eta |} \right| (\xi,\eta) \le C_b |\xi|^{s-1}|\eta|^{s} |\xi-\eta|^2 ( |\xi|^{1-2b} + |\xi-\eta|^{1-2b} ).
		\end{split}
		\end{equation*}  Next, we have similarly in the case $|\xi|>|\eta|$ that \begin{equation*}
		\begin{split}
		\left| G_0' \mathbf{1}_{ |\xi|>|\eta |} \right| (\xi,\eta) \le C_b |\xi|^{s+1-\frac{3}{2}b}|\eta|^{s-1-\frac{b}{2}} |\xi-\eta|^2 
		\end{split}
		\end{equation*} and when $b<\frac{2}{3}$, we can further bound \begin{equation*}
		\begin{split}
		\left| G_0' \mathbf{1}_{ |\xi|>|\eta |} \right| (\xi,\eta) \le C_b |\xi|^{s}|\eta|^{s-1-\frac{b}{2}}  |\xi-\eta|^2 ( |\eta|^{1-\frac{3}{2}b} + |\xi-\eta|^{1-\frac{3}{2}b} ). 
		\end{split}
		\end{equation*} }
		Therefore, for the whole range of $0<b<1$, we obtain the bound \begin{equation*}
		\begin{split}
		|G_0'(\xi,\eta)| \le C_b (1+|\xi|)^s (1+|\eta|)^s |\xi-\eta|^2 ( 1 + |\xi-\eta|),
		\end{split}
		\end{equation*} which results in the estimate \begin{equation*}
		\begin{split}
		\left| \calT[G_0] - c_d\int_{\R^d} \rho |\Lmb^{s-b}\nb\rho|^2\,dx  \right| \le C\nrm{ |\xi|^2 ( 1 + |\xi|)\whr }_{L^1_\xi} \nrm{\rho}_{H^{s}}^2 .
		\end{split}
		\end{equation*} (The constant $C_b>0$ is uniform in the range $b\in[0,1]$, so that the endpoint cases $b = 0$ and $b = 1$ can be covered as well.)

		\item $G_1 = |\xi|^s (s\eta\cdot(\xi-\eta)) (\xi\cdot\eta)|\eta|^{s-2-2b}$. We write \begin{equation*}
		\begin{split}
		G_1 &= |\xi|^s (s\eta\cdot(\xi-\eta)) (\xi-\eta)\cdot\eta|\eta|^{s-2-2b} - \frac{s}{2}|\xi|^s   |\eta|^{s -2b} |\xi-\eta|^2 + \frac{s}{2} |\xi|^s |\eta|^{s-2b} (\eta+\xi)\cdot (\xi-\eta) \\
		& =: G_{11} + G_{12} + G_{13}
		\end{split}
		\end{equation*} so that \begin{equation*}
		\begin{split}
		\left| \calT[G_{11} + G_{12}] \right| \le \nrm{ |\xi|^2 \whr }_{L^1_\xi} \nrm{\rho}_{H^{s}}^2 .
		\end{split}
		\end{equation*} For $G_{13}$, we proceed similarly as in the above; writing \begin{equation*}
		\begin{split}
		G_{13} &= \frac{s}{2} |\xi|^{s-b} |\eta|^{s-b} (\eta+\xi)\cdot (\xi-\eta) + \frac{s}{2} |\xi|^{s-b} |\eta|^{s-2b} (|\xi|^b - |\eta|^b) (\eta+\xi)\cdot (\xi-\eta) \\
		& =: G_{131} + G_{132},
		\end{split}
		\end{equation*} we have $\calT[G_{131}] = 0$ by anti-symmetry and then using \eqref{eq:b-b}, we can estimate \begin{equation*}
		\begin{split}
		\left| \calT[G_{132}] \right| \le  C\nrm{ |\xi|^2 ( 1 + |\xi|)\whr }_{L^1_\xi} \nrm{\rho}_{H^{s}}^2 .
		\end{split}
		\end{equation*} We omit the details, since the proof is completely parallel to the case of $G_0'$ (consider separately the cases $|\xi|\le |\eta|$ and $|\xi|>|\eta|$). The point is that once a factor of $|\xi-\eta|^2$ is extracted, we can bound the remaining factor using $(1+|\xi-\eta|)(1+|\xi|)^s(1+|\eta|)^s$. We have that \begin{equation*}
		\begin{split}
		\left| \calT[G_1] \right| \le C\nrm{ |\xi|^2 ( 1 + |\xi|)\whr }_{L^1_\xi} \nrm{\rho}_{H^{s}}^2 .
		\end{split}
		\end{equation*}
	\end{itemize}

	\medskip

	\noindent {Before proceeding further, we present an auxiliary lemma below whose proof will be given later for a smooth flow of reading.
	\begin{lemma}\label{lem_aux} Let $s \geq 3$. For vectors $\xi, \eta \in \R^d$, there exists $C>0$ independent of $\xi$ and $\eta$ such that 
	\begin{equation}\label{eq:elementary}
	\begin{split}
	\left| |\xi|^s - |\xi-\eta|^s - |\eta|^s - s\eta\cdot(\xi-\eta)|\eta|^{s-2} \right| \le C ( |\xi-\eta|^2|\eta|^{s-2} + |\eta||\xi-\eta|^{s-1} ).
	\end{split}
	\end{equation}
	\end{lemma}
	}

	We are in a position to close an a priori estimate. Recall from the above that \begin{equation*}
	\begin{split}
	\frac{d}{dt} \frac{c_d}{2} \nrm{\rho}_{\dot{H}^s}^2  =  c_K \calT[G], \quad G := |\xi|^{2s} \xi\cdot\eta |\eta|^{-2b}.
	\end{split}
	\end{equation*}
	 {
	Then making use of Lemma \ref{lem_aux} and recalling the definition of $G_0, G_1$, and $G_s$ give }
	\begin{equation*}
	\begin{split}
	&\left| G - G_0 - G_1 - G_s \right|(\xi,\eta) \\
	&\qquad \le C( |\xi-\eta|^2|\eta|^{s-2} + |\eta||\xi-\eta|^{s-1} )|\xi|^{s}|\eta|^{1-2b} (|\xi-\eta|+|\eta|) \\
	&\qquad \le   { C|\xi|^s |\eta|^{2-2b} |\xi-\eta|^2 (|\eta|^{s-2} + |\eta|^{s-3}|\xi-\eta|+|\eta||\xi-\eta|^{s-3}+ |\xi-\eta|^{s-2}). } \\
	&\qquad \le C|\xi|^s(|\eta|^{s-2b}|\xi-\eta|^2 + |\eta|^{s-1-2b}|\xi-\eta|) + C|\xi|^s(|\xi-\eta|^s| |\eta|^{2-2b} + |\xi-\eta|^{s-1}|\eta|^{3-2b}).
	\end{split}
	\end{equation*} From this, we obtain that \begin{equation}\label{eq:Y}
	\begin{split}
	&\left| \calT[G - G_0 - G_1 - G_s  ] \right| \\
	&\qquad \le C\iint_{\R^d \times \R^d} |\xi|^s(|\eta|^{s-2b}|\xi-\eta|^2 + |\eta|^{s-1-2b}|\xi-\eta|) |\whr(\eta)| |\whr(\xi)||\whr(\eta-\xi)| \,d\eta\, d\xi  \\
	&\qquad\quad  + C\iint_{\R^d \times \R^d} |\xi|^s(|\xi-\eta|^s| |\eta|^{2-2b} + |\xi-\eta|^{s-1}|\eta|^{3-2b}) |\whr(\eta)| |\whr(\xi)||\whr(\eta-\xi)| \,d\eta\, d\xi .
	\end{split}
	\end{equation} We see that both terms on the right hand side is bounded by \begin{equation*}
	\begin{split}
	 {C\nrm{ |\xi|^2(1+|\xi|)\whr }_{L^1_\xi} }\nrm{\rho}_{H^{s}}^2,
	\end{split}
\end{equation*} first using H\"older's inequality in the $\xi$ variable and then applying Young's convolution inequality. Then, recalling the estimates for $\calT[G_0], \calT[G_1]$, and $\calT[G_s]$, we have that \begin{equation*}
	\begin{split}
	\left| \frac{d}{dt} \frac{1}{2} \nrm{\rho}_{\dot{H}^s}^2 -c_K\int_{\R^d} \rho |\Lmb^{s-b}\nb\rho|^2 \,dx\right| \le C\nrm{ |\xi|^2 ( 1 + |\xi|)\whr }_{L^1_\xi} \nrm{\rho}_{H^{s}}^2 .
	\end{split}
	\end{equation*} On the other hand, we have trivially \begin{equation}\label{eq:L2-apriori}
	\begin{split}
	\left| \frac{d}{dt} \frac{1}{2} \nrm{\rho}_{L^2}^2 \right| \le C\nrm{ |\xi|^2 ( 1 + |\xi|)\whr }_{L^1_\xi} \nrm{\rho}_{H^{s}}^2 .
	\end{split}
	\end{equation} Therefore, if $c_K\le 0$ and $s > \frac{d}{2} + 3$, we conclude the estimate \begin{equation}\label{eq:case1-concl}
	\begin{split}
	\left|\frac{d}{dt} \nrm{\rho}_{H^s}^2\right| \le C \nrm{\rho}_{H^s}^3,
	\end{split}
	\end{equation} assuming that $\rho\ge0$.
	
	\medskip
	
	\noindent\textit{Case 2. $\nu>0$}
	
	\medskip 
	
	\noindent As in Case 1, we consider the $\dot{H}^s$ estimate for the solution $\rho$, assuming for simplicity that $0<b<1$. Applying $\Lmb^s$ to the equation and integrating against $\Lmb^s \rho$, we obtain
	\begin{equation*}
	\begin{split}
	\frac{c_d}{2}\frac{d}{dt} \int_{\R^d} |\Lmb^s\rho|^2\,dx + \nu \nrm{\rho}_{\dot{H}^{s+1}}^2 = c_K \calT[G] 
	\end{split}
	\end{equation*} with $G$ as in the above. Here $c_d>0$ is a constant depending only on $d$. We again decompose \begin{equation*}
	\begin{split}
	G = (G_0-G_0') + G_0' + G_1 + G_s + (G-G_0-G_1-G_s)
	\end{split}
	\end{equation*} and the term corresponding to $G_0-G_0'$ is now handled as follows: 
	 \begin{equation*}
	\begin{split}
	|c_K|\int_{\R^d} \rho |\Lmb^{s-b} \nb\rho|^2\, dx \le C\nrm{\rho}_{L^\infty} \nrm{\rho}_{\dot{H}^s}^{2b} \nrm{\rho}_{\dot{H}^{s+1}}^{2(1-b)} \le C_\nu  \nrm{\rho}_{L^\infty}^{\frac{1}{b}} \nrm{\rho}_{\dot{H}^s}^2  + \frac{\nu}{8} \nrm{\rho}_{\dot{H}^{s+1}}^2.
	\end{split}
	\end{equation*} Next, we use a rough bound \begin{equation*}
	\begin{split}
	|G_0'(\xi,\eta)| \le C|\xi|^{s-\frac{3}{2}b+1} |\eta|^{s-2b+1} (|\xi|^{\frac{3}{2}b-2} + |\eta|^{\frac{3}{2}b-2})|\xi-\eta|^2 \le C |\xi-\eta|^2 (1 + |\xi|^s)(1 + |\eta|^s)(|\xi|+|\eta|),
	\end{split} 
	\end{equation*} which gives \begin{equation*}
	\begin{split}
	\left| \calT[G_0'] \right| \le  C\nrm{ |\xi|^2\whr }_{L^1_\xi} \nrm{\rho}_{H^{s}}(\nrm{\rho}_{\dot{H}^{s+1}} + \nrm{\rho}_{L^2}) \le C_\nu(1 + \nrm{ |\xi| \whr }_{L^1_\xi}^2)\nrm{\rho}_{{H}^s}^2  + \frac{\nu}{8} \nrm{\rho}_{\dot{H}^{s+1}}^2. 
	\end{split}
	\end{equation*} For $G_s$, we simply use the previous bound \eqref{eq:TG1}, and regarding $G_1$, we observe  that \begin{equation*}
	\begin{split}
	|G_1|\le C|\xi-\eta||\eta|^{s-2b}|\xi|^{s+1} 
	\end{split}
	\end{equation*} so that arguing as in the proof following \eqref{eq:Y} above, \begin{equation*}
	\begin{split}
	\left| \calT[G_1] \right| + \left| \calT[G_s] \right| \le C\nrm{ (1 + |\xi|)\whr }_{L^1_\xi}\nrm{\rho}_{H^{s}} \nrm{\rho}_{\dot{H}^{s+1}} \le C_\nu(1 + \nrm{ |\xi|\whr }_{L^1_\xi}^2)\nrm{\rho}_{{H}^s}^2  + \frac{\nu}{8} \nrm{\rho}_{\dot{H}^{s+1}}^2.
	\end{split}
	\end{equation*}  Finally, using \begin{equation*}
	\begin{split}
	|G - G_0 - G_1 - G_s | \le C(|\xi-\eta|^2|\eta|^{s-2}+|\eta||\xi-\eta|^{s-1})|\xi|^{s+1}|\eta|^{1-2b} 
	\end{split}
	\end{equation*} we estimate \begin{equation*}
	\begin{split}
		\left| \calT[G - G_0 - G_1 - G_s  ] \right| &\le {C\nrm{ |\xi|(1+|\xi|)\whr }_{L^1_\xi} }\nrm{\rho}_{H^{s}}(\nrm{\rho}_{\dot{H}^{s+1}}  +\nrm{\rho}_{L^2})\\
		&\le C_\nu(1 + \nrm{ |\xi|(1+|\xi|)\whr }_{L^1_\xi}^2)\nrm{\rho}_{{H}^s}^2  + \frac{\nu}{8} \nrm{\rho}_{\dot{H}^{s+1}}^2.
	\end{split}
	\end{equation*} Combining the estimates, we obtain that \begin{equation*}
	\begin{split}
	\left|\frac{c_d}{2}\frac{d}{dt} \int |\Lmb^s\rho|^2 dx + \frac{\nu}{2} \nrm{\rho}_{\dot{H}^{s+1}}^2 \right| \le  C_\nu(1 + \nrm{ |\xi|(1+|\xi|)\whr }_{L^1_\xi}^2)\nrm{\rho}_{{H}^s}^2\le C_\nu(1 + \nrm{\rho}_{H^s}^2)\nrm{\rho}_{H^s}^2 . 
	\end{split}
	\end{equation*} Together with the $L^2$ estimate for $\rho$ given in \eqref{eq:L2-apriori}, we conclude \begin{equation}\label{eq:case2-concl}
	\begin{split}
	\frac{d}{dt} \nrm{\rho}_{H^s}^2 \le C_\nu(1 + \nrm{\rho}_{H^s}^2)\nrm{\rho}_{H^s}^2 . 
	\end{split}
	\end{equation} 
	
	\medskip
	
	\noindent\textit{Existence and uniqueness}
	
	\medskip 
	
	\noindent We prove existence and uniqueness for the case $c_K<0$ and $\nu=0$. The proof for the other case can be carried out in a similar manner. We shall only sketch the proof since this procedure is rather standard (cf. \cite{K_inv,KL}), given the $H^s$ a priori estimate. To begin with, we consider the following regularized system \begin{equation}\label{eq:hyper}
	\begin{split}
	&\rd_t\rho^{(\mu)} + \nb\cdot ( \rho^{(\mu)} u^{(\mu)} ) = 0, \\
	&u^{(\mu)} = c_K \Lmb^{\alp-d}_\mu \nb \rho^{(\mu)},\\
	&\rho^{(\mu)}(t=0) = \phi^{(\mu)} * \rho_0
	\end{split}
	\end{equation}  for each $\mu>0$. Here, $\Lmb_\mu^{-b}$ for $b>0$ is a regularization of $\Lmb^{-b}$ defined by the multiplier \begin{equation*}
	\begin{split}
		|\xi|^{-b}\chi(\mu|\xi|)
	\end{split}
\end{equation*} where $\chi(\cdot)\ge0$ is a smooth function supported in $[0,1]$ and satisfying $\chi(0)=1$, so that we have $\Lmb_\mu^{-b}\rightarrow\Lmb^{-b}$ as $\mu\rightarrow 0^+$. Note that for a given $\mu>0$, $u^{(\mu)} = c_K \Lmb^{\alp-d}_\mu \nb \rho^{(\mu)}$ is $C^\infty$-smooth for $\rho^{(\mu)}\in L^2(\Omg)$. Moreover, $\{\phi^{(\mu)}\}$ is an approximation of the identity; $\phi^{(\mu)}(\cdot) = \frac{1}{\mu^d} \phi(\frac{\cdot}{\mu})$ for some non-negative smooth bump function $\phi$ satisfying $\int_{\Omg}\phi\,dx = 1$. Since $\rho_0\ge0$, we have that $\rho^{(\mu)}_0\ge 0$ as well. 

It is not difficult to show that for each $\mu>0$, there is a local-in-time $C^\infty$-smooth solution $\rho^{(\mu)}\ge 0$ in $[0,T_\mu]$ for $T_\mu>0$ possibly depending on $\mu$. For instance, one can consider the iterations \begin{equation*}
	\begin{split}
		\rd_t\rho^{(\mu),n} + \nb\cdot ( \rho^{(\mu),n} u^{(\mu),n-1} ) = 0,\\
		u^{(\mu),n-1} = c_K \Lmb^{\alp-d}_\mu \nb \rho^{(\mu),n-1}
	\end{split}
\end{equation*} with $\rho^{(\mu),0} \equiv \rho^{(\mu)}_0$ on $[0,T_\mu]$. By taking $T_\mu>0$ sufficiently small (depending on $\mu$ and the initial data), the iterates $\{ \rho^{(\mu),n} \}_{n\ge 0}$ are non-negative and uniformly bounded in $C([0,T_\mu];C^1(\Omg))$. Therefore, we have uniform convergence $\rho^{(\mu),n} \rightarrow \rho^{(\mu)}$ for some smooth function $\rho^{(\mu)}\ge0$, which provides a solution to \eqref{eq:hyper}.

Furthermore, the a priori estimate \eqref{eq:case1-concl} can be justified for $\rho^{(\mu)}$ with any $\mu>0$ (by repeating the proof with $\Lmb$ replaced by $\Lmb_\mu$). Therefore, for any $\mu>0$, the local solution $\rho^{(\mu)}$ can be extended to some $T>0$ which depends only on $\nrm{\rho_0}_{H^s}$, with uniformly bounded norm $\nrm{ \rho^{(\mu)}}_{L^\infty([0,T];H^s)}$. Therefore, by passing to a subsequence $\{\mu_k\}$ with $\mu_k\rightarrow 0^+$, there exists some $\rho\in L^\infty([0,T];H^s(\Omg))$ such that we have convergence \begin{equation*}
	\begin{split}
	\rho^{(\mu_k)} \longrightarrow \rho 
	\end{split}
	\end{equation*} strongly in $L^\infty([0,T];H^{s-1}(\Omg))$ and weakly in $L^\infty([0,T];H^s(\Omg)).$ First, strong convergence in $L^\infty([0,T];H^{s-1}(\Omg))$ implies pointwise convergence, which in particular guarantees that $\rho\ge0$.
	Next, it is straightforward to see that $\rho$ is a solution to \eqref{eq:Euler-first} with initial data $\rho_0$. It still remains to show that $\rho$ belongs to $C([0,T];H^s(\Omg))$. To this end, we shall prove continuity at $t =0$; the case $t>0$ can be treated in a similar way. First, observe that $\rho(t_k) \rightharpoonup \rho_0$ weakly in $H^s$ for $t_k\rightarrow 0^+$. This follows from strong convergence $\rho^{(\mu_k)} \rightarrow \rho$ in $L^\infty([0,T]; L^2(\Omg))$, $\rho^{(\mu_k)} \in C([0,T];H^s(\Omg))$, and that $L^2(\Omg)$ is dense in the dual of $H^s(\Omg)$. From weak convergence, it follows that \begin{equation*}
	\begin{split}
	\nrm{\rho_0}_{H^s} \le \liminf_{k\rightarrow\infty}  \nrm{\rho(t_k)}_{H^s}.  
	\end{split}
	\end{equation*} On the other hand, for any sufficiently small $t_k$, \begin{equation*}
	\begin{split}
	\nrm{\rho(t_k)}_{H^s} & \le \limsup_{\mu\rightarrow 0} \nrm{\rho^{(\mu)}(t_k)}_{H^s} \le \limsup_{\mu\rightarrow 0} \frac{ \nrm{\rho^{(\mu)}_0}_{H^s} }{1 - Ct_k\nrm{\rho_0^{(\mu)}}_{H^s} }
	\end{split}
	\end{equation*} where $C>0$ is some absolute constant. In the first inequality, we used the fact that the norm can only decrease in the weak limit. The second inequality follows from the a priori estimate for $\rho^{(\mu)}$ in $H^s(\Omg)$. Now, taking the limit $k \rightarrow \infty$ and recalling that $\rho_0^{(\mu)}$ converges strongly in $H^s(\Omg)$ to $\rho_0$, we deduce \begin{equation*}
	\begin{split}
	\limsup_{k\rightarrow\infty} \nrm{\rho(t_k)}_{H^s} & \le \nrm{\rho_0}_{H^s} .
	\end{split}
	\end{equation*} This shows that the map $t\mapsto \nrm{\rho(t)}_{H^s}$ is continuous at $t = 0$. From weak convergence in $H^s(\Omg)$, it follows that $t\mapsto \rho(t)$ is continuous at $t = 0$ with values in $H^s$ as well. 
	
	To prove uniqueness, we assume that there exist two non-negative solutions $\rho$ and $\trho$ belonging to $L^\infty([0,T];H^s(\Omg))$ for some $T>0$. Defining $g = \rho-\trho$ and $v = u -\tu$, we have that the equation for $g$ is given by 
	\begin{equation*}
	\begin{split}
	\rd_t g + \nb\cdot ( g u ) + \nb\cdot(\trho v) = 0.  
	\end{split}
	\end{equation*} Multiplying by $g$ and integrating, it is not difficult to derive the following estimate: \begin{equation*}
	\begin{split}
	\frac{d}{dt} \nrm{g}_{L^2}^2 \le C\left(\nrm{\nb u}_{L^\infty} \nrm{g}_{L^2}^2 + \nrm{\nb\trho}_{L^\infty}\nrm{g}_{L^2}\nrm{g}_{H^1} + \nrm{\trho}_{L^\infty}\nrm{g}_{H^1}^2 \right).
	\end{split}
	\end{equation*} To close the estimate, we consider the equation for $\rd g$: \begin{equation*}
	\begin{split}
	\frac{1}{2}\frac{d}{dt}\nrm{\rd g}_{L^2}^2 + \int_{\Omg} \nb \cdot (\rd g u) \rd g dx + \int_{\Omg} \nb\cdot (g \rd u)\rd g dx + \int_{\Omg} \nb\cdot ( \rd\trho v ) \rd g dx + \int_{\Omg} \nb\cdot (\trho \rd v) \rd g dx = 0. 
	\end{split}
	\end{equation*} First, it is not difficult to estimate \begin{equation*}
	\begin{split}
	\left| \int_{\Omg} \nb \cdot (\rd g u) \rd g dx  \right| + \left|\int_{\Omg} \nb\cdot (g \rd u)\rd g dx  \right| \le C \nrm{\nb u}_{L^\infty} \nrm{g}_{H^1}^2\le C \nrm{\rho}_{H^s} \nrm{g}_{H^1}^2
	\end{split}
	\end{equation*} with an integration by parts. Next, we claim that the other terms can be bounded as follows: \begin{equation*}
	\begin{split}
	\left| \int_{\Omg} \nb\cdot ( \rd\trho v ) \rd g dx + \int_{\Omg} \nb\cdot (\trho \rd v) \rd g dx  + c_K\int_{\Omg}\trho |\nb\rd \Lmb^{-b}g|^2 dx \right| \le C\nrm{\trho}_{H^s}\nrm{g}_{H^1}^2 
	\end{split}
	\end{equation*} where $-2b=\alp-d$. To see this, we consider the term with most derivatives on $v$: \begin{equation*}
	\begin{split}
	\int_{\Omg} \trho \rd \nb\cdot v \rd g dx =  c_K\int_{\Omg} \trho \nb \cdot \Lmb^{-b} ( \Lmb^{-b} \nb\rd g )    \rd g dx .
	\end{split}
	\end{equation*} Hence\begin{equation*}
	\begin{split}
	&\int_{\Omg} \trho \rd \nb\cdot v \rd g dx + c_K\int_{\Omg}\trho |\nb\rd \Lmb^{-b}g|^2 dx \cr
	&\quad = c_K\int_{\Omg} [\trho\nb \cdot,\Lmb^{-b}]  ( \Lmb^{-b} \nb\rd g ) \, \rd g dx  - c_K \int_{\Omg} \nb\trho \cdot  ( \Lmb^{-b} \nb\rd g ) \Lmb^{-b}\rd g\, dx. 
	\end{split}
	\end{equation*} It is easy to see that the last term is bounded by $C\nrm{\trho}_{H^s}\nrm{g}_{H^1}^2 $, and the first term on the right hand side can be estimated in a parallel manner with the corresponding term from the a priori estimate above (cf. \eqref{eq:comm}). Therefore, we obtain that \begin{equation*}
	\begin{split}
	\left| \frac{1}{2}\frac{d}{dt}\nrm{g}_{H^1}^2 - c_K\int_{\Omg}\trho |\nb\rd \Lmb^{-b}g|^2 dx \right| \le C( \nrm{\trho}_{H^s}+\nrm{\rho}_{H^s})\nrm{g}_{H^1}^2. 
	\end{split}
	\end{equation*} Since $c_K<0$, this shows that if $\nrm{g}_{H^1} = 0$ at $t = 0$, $\nrm{g}_{H^1} = 0$ as long as $\nrm{\trho}_{H^s}+\nrm{\rho}_{H^s}$ is bounded. This gives uniqueness.  
\end{proof}

\noindent{We finally provide a proof of Lemma \ref{lem_aux} which is given in \cite{CJO}.}
\begin{proof}[Proof of Lemma \ref{lem_aux}]
	Using the mean value theorem, we begin with writing \begin{equation*}
		\begin{split}
			|\xi|^s - |\eta|^s  = -s(\eta-\xi) \cdot  \int_0^1 |A(\rho)|^{s-2} A(\rho)\, d\rho ,
		\end{split}
	\end{equation*} and \begin{equation}\label{eq:con}
	\begin{split}
		|\xi|^s - |\eta|^s - |\xi-\eta|^s  = -s(\eta-\xi) \cdot\int_0^1 \left(|A(\rho)|^{s-2} A(\rho) - |B(\rho)|^{s-2} B(\rho) \right) d\rho,
	\end{split}
\end{equation}  where $A$ and $B$ are defined respectively by  \begin{equation*}
		\begin{split}
			A(\rho) = (1-\rho)\xi+\rho \eta, \quad B(\rho) = (1-\rho)(\xi-\eta). 
		\end{split}
	\end{equation*} Now, we write  \begin{equation*}
		\begin{split}
			|A(\rho)|^{s-2} A(\rho) - |B(\rho)|^{s-2} B(\rho) &= |B(\rho)|^{s-2} (A(\rho) -  B(\rho)) + (|A(\rho)|^{s-2}- |B(\rho)|^{s-2}) A(\rho) \\
			& =: I + II. 
		\end{split}
	\end{equation*}We then compute that   
	\begin{equation*}
		\begin{split}
			-s(\eta-\xi) \cdot \int_0^1 I \, d\rho = \frac{s}{s-1} (\xi-\eta)\cdot \eta |\xi-\eta|^{s-2}
		\end{split}
	\end{equation*} and \begin{equation*}
	\begin{split}
		-s(\eta-\xi) \cdot \int_0^1 II \,d\rho = -s(s-2)(\eta-\xi)\cdot\int_0^1\int_0^1 A(\rho) |E(\rho,\sigma)|^{s-4} \eta \cdot E(\rho,\sigma)\,d\rho d\sigma,
	\end{split}
\end{equation*} 
where we define $E(\rho,\sigma) = (1-\rho)\xi + (\rho-\sigma)\eta$. Therefore, recalling \eqref{eq:con} and the definition of $I$ and $II$, \begin{equation*}
	\begin{split}
		|\xi|^s - |\eta|^s - |\xi-\eta|^s  &= \frac{s}{s-1} (\xi-\eta)\cdot \eta |\xi-\eta|^{s-2}  \cr
		&\quad -s(s-2)(\eta-\xi)\cdot\int_0^1\int_0^1 A(\rho) |E(\rho,\sigma)|^{s-4} \eta \cdot E(\rho,\sigma)\,d\rho d\sigma,
	\end{split}
\end{equation*} so that after subtracting $s(\xi-\eta)\cdot\eta |\xi-\eta|^{s-2}$ from both sides, \begin{equation*}
		\begin{split}
			&|\xi|^s - |\eta|^s - |\xi-\eta|^s - s(\xi-\eta)\cdot\eta |\xi-\eta|^{s-2} \\
			&\quad  = s(s-2)(\xi-\eta) \cdot \left[ \int_0^1 \int_0^1  A(\rho)|E(\rho,\sigma)|^{s-4}\eta\cdot E(\rho,\sigma) - \frac{1}{s-1} \eta |\xi-\eta|^{s-2}\,d\sigma d\rho \right] .
		\end{split}
	\end{equation*} The last integral vanishes when $\eta=0$. Therefore, we obtain the estimate \begin{equation*}
		\begin{split}
			\left||\xi|^s - |\eta|^s - |\xi-\eta|^s - s(\xi-\eta)\cdot\eta |\xi-\eta|^{s-2} \right| &\le C_s|\xi-\eta| |\eta|^2 \left( |\xi-\eta|^{s-3} + |\eta|^{s-3} \right).
		\end{split}
	\end{equation*} Switching the roles of $\xi-\eta$ and $\eta$ gives \eqref{eq:elementary}.
\end{proof}

\section*{Acknowledgement} We thank Sung-Jin Oh for helpful conversations regarding degenerate parabolic equations and pseudo-differential calculus. 
YPC has been supported by NRF grant (No. 2017R1C1B2012918), POSCO Science Fellowship of POSCO TJ Park Foundation, and Yonsei University Research Fund of 2019-22-021 and 2020-22-0505. IJJ has been supported  by a KIAS Individual Grant MG066202 at Korea Institute for Advanced Study, the Science Fellowship of POSCO TJ Park Foundation, and the National Research Foundation of Korea grant (No. 2019R1F1A1058486). 

\bibliographystyle{amsplain}
\bibliography{FPME}

\providecommand{\bysame}{\leavevmode\hbox to3em{\hrulefill}\thinspace}
\providecommand{\MR}{\relax\ifhmode\unskip\space\fi MR }
\providecommand{\MRhref}[2]{%
  \href{http://www.ams.org/mathscinet-getitem?mr=#1}{#2}
}
\providecommand{\href}[2]{#2}
\begin{thebibliography}{10}

\bibitem{BIK15}
Piotr Biler, Cyril Imbert, and Grzegorz Karch, \emph{The nonlocal porous medium
  equation: {B}arenblatt profiles and other weak solutions}, Arch. Ration.
  Mech. Anal. \textbf{215} (2015), no.~2, 497--529. \MR{3294409}

\bibitem{CSV}
Luis Caffarelli, Fernando Soria, and Juan~Luis V\'{a}zquez, \emph{Regularity of
  solutions of the fractional porous medium flow}, J. Eur. Math. Soc. (JEMS)
  \textbf{15} (2013), no.~5, 1701--1746. \MR{3082241}

\bibitem{CaVa}
Luis Caffarelli and Juan~Luis Vazquez, \emph{Nonlinear porous medium flow with
  fractional potential pressure}, Arch. Ration. Mech. Anal. \textbf{202}
  (2011), no.~2, 537--565. \MR{2847534}

\bibitem{CV11}
Luis~A. Caffarelli and Juan~Luis V\'{a}zquez, \emph{Asymptotic behaviour of a
  porous medium equation with fractional diffusion}, Discrete Contin. Dyn.
  Syst. \textbf{29} (2011), no.~4, 1393--1404. \MR{2773189}

\bibitem{CCCSS16}
J.~Calvo, J.~Campos, V.~Caselles, O.~S\'{a}nchez, and J.~Soler, \emph{Pattern
  formation in a flux limited reaction-diffusion equation of porous media
  type}, Invent. Math. \textbf{206} (2016), no.~1, 57--108. \MR{3556525}

\bibitem{CHSV15}
J.~A. Carrillo, Y.~Huang, M.~C. Santos, and J.~L. V\'{a}zquez,
  \emph{Exponential convergence towards stationary states for the 1{D} porous
  medium equation with fractional pressure}, J. Differential Equations
  \textbf{258} (2015), no.~3, 736--763. \MR{3279352}

\bibitem{CCY19}
Jos\'{e}~A. Carrillo, Katy Craig, and Yao Yao, \emph{Aggregation-diffusion
  equations: dynamics, asymptotics, and singular limits}, Active particles,
  {V}ol. 2, Model. Simul. Sci. Eng. Technol., Birkh\"{a}user/Springer, Cham,
  2019, pp.~65--108. \MR{3932458}

\bibitem{CMV03}
Jos\'{e}~A. Carrillo, Robert~J. McCann, and C\'{e}dric Villani, \emph{Kinetic
  equilibration rates for granular media and related equations: entropy
  dissipation and mass transportation estimates}, Rev. Mat. Iberoamericana
  \textbf{19} (2003), no.~3, 971--1018. \MR{2053570}

\bibitem{CCCGW}
Dongho Chae, Peter Constantin, Diego C\'{o}rdoba, Francisco Gancedo, and
  Jiahong Wu, \emph{Generalized surface quasi-geostrophic equations with
  singular velocities}, Comm. Pure Appl. Math. \textbf{65} (2012), no.~8,
  1037--1066. \MR{2928091}

\bibitem{CJO}
Dongho Chae, In-Jee Jeong, and Sung-Jin Oh, \emph{Pseudo-differential calculus,
  degenerating wavepackets, and illposedness for active scalars with singular
  drift}, in preparation.

\bibitem{CJ-limit}
Young-Pil. Choi and In-Jee Jeong, \emph{The relaxation to the fractional porous
  medium equation from the {E}uler-{R}iesz system}, in preparation.

\bibitem{GL1}
Giambattista Giacomin and Joel~L. Lebowitz, \emph{Phase segregation dynamics in
  particle systems with long range interactions. {I}. {M}acroscopic limits}, J.
  Statist. Phys. \textbf{87} (1997), no.~1-2, 37--61. \MR{1453735}

\bibitem{GL2}
\bysame, \emph{Phase segregation dynamics in particle systems with long range
  interactions. {II}. {I}nterface motion}, SIAM J. Appl. Math. \textbf{58}
  (1998), no.~6, 1707--1729. \MR{1638739}

\bibitem{GL4}
Giambattista Giacomin, Joel~L. Lebowitz, and Rossana Marra, \emph{Macroscopic
  evolution of particle systems with short- and long-range interactions},
  Nonlinearity \textbf{13} (2000), no.~6, 2143--2162. \MR{1794850}

\bibitem{GL3}
Giambattista Giacomin, Joel~L. Lebowitz, and Errico Presutti,
  \emph{Deterministic and stochastic hydrodynamic equations arising from simple
  microscopic model systems}, Stochastic partial differential equations: six
  perspectives, Math. Surveys Monogr., vol.~64, Amer. Math. Soc., Providence,
  RI, 1999, pp.~107--152. \MR{1661764}

\bibitem{K_inv}
Tosio Kato, \emph{Nonstationary flows of viscous and ideal fluids in {${\bf
  R}\sp{3}$}}, J. Functional Analysis \textbf{9} (1972), 296--305. \MR{0481652}

\bibitem{KL}
Tosio Kato and Chi~Yuen Lai, \emph{Nonlinear evolution equations and the
  {E}uler flow}, J. Funct. Anal. \textbf{56} (1984), no.~1, 15--28. \MR{735703}

\bibitem{LT04}
Hailiang Li and Giuseppe Toscani, \emph{Long-time asymptotics of kinetic models
  of granular flows}, Arch. Ration. Mech. Anal. \textbf{172} (2004), no.~3,
  407--428. \MR{2062430}

\bibitem{LMS18}
Stefano Lisini, Edoardo Mainini, and Antonio Segatti, \emph{A gradient flow
  approach to the porous medium equation with fractional pressure}, Arch.
  Ration. Mech. Anal. \textbf{227} (2018), no.~2, 567--606. \MR{3740382}

\bibitem{ME99}
Alexander Mogilner and Leah Edelstein-Keshet, \emph{A non-local model for a
  swarm}, J. Math. Biol. \textbf{38} (1999), no.~6, 534--570. \MR{1698215}

\bibitem{STV19}
Diana Stan, F\'{e}lix del Teso, and Juan~Luis V\'{a}zquez, \emph{Existence of
  weak solutions for a general porous medium equation with nonlocal pressure},
  Arch. Ration. Mech. Anal. \textbf{233} (2019), no.~1, 451--496. \MR{3974645}

\bibitem{Va17}
Juan~Luis V\'{a}zquez, \emph{The mathematical theories of diffusion: nonlinear
  and fractional diffusion}, Nonlocal and nonlinear diffusions and
  interactions: new methods and directions, Lecture Notes in Math., vol. 2186,
  Springer, Cham, 2017, pp.~205--278. \MR{3588125}

\end{thebibliography}

\end{document}